\documentclass[a4paper, 11pt]{article}%
\usepackage{amsfonts,amsmath,amsthm,amssymb,stmaryrd}
\usepackage{graphicx, tipa}
\usepackage{authblk}
\usepackage{wasysym, url}
\usepackage[utf8]{inputenc}

\author{Nicolae Mihalache}
\affil{Univ Paris Est Creteil, Univ Gustave Eiffel, CNRS, LAMA UMR8050, F-94010 Creteil, France}
\date{}

\theoremstyle{plain}

\newtheorem{thm}{Theorem}
\newtheorem*{sa}{Simplifying assumption}

\newtheorem{definition}[thm]{Definition}

\newtheorem{lemma}[thm]{Lemma}

\newtheorem{proposition}[thm]{Proposition}

\def\R{{\mathbb R}}
\def\D{{\mathbb D}}
\def\C{{\mathbb C}}
\def\F{{\mathcal F}}
\def\N{{\mathbb N}}

\def\Om{{\Omega}}
\def\pa{{\partial}}
\def\si{{\sigma}}
\def\al{{\alpha}}

\def\Si{{\Sigma}}
\def\th{{\theta}}
\def\ga{{\gamma}}

\def\F{{\mathcal F}}
\def\A{{\mathcal A}}

\def\pf{{\pi_\F}}
\def\se{{\subseteq}}
\def\sm{{\setminus}}

\def\es{{\emptyset}}
\def\ol{\overline}

\newcommand{\arc}[1]{%
  \setbox9=\hbox{#1}%
  \ooalign{\resizebox{\wd9}{\height}{\texttoptiebar{\phantom{A}}}\cr#1}}

\def\ra{{\rightarrow}}

\def\ve{{\varepsilon}}
\def\de{{\delta}}

\def\dist{{\mathrm {dist}}}
\def\diam{{\mathrm {diam\,}}}

\newcommand\rpro[1]{{Proposition \ref{prop#1}}}

\begin{document}

\title{On the local connectivity of attractors of Markov IFS}
\maketitle{}

\begin{abstract}We prove an extension of M. Hata's theorem \cite{HAT} for planar Markov Iterated Function Systems satisfying a strong version of the Open Set Condition. More precisely, if the attractor of such a system is connected, then it is locally connected. We construct counterexamples to show that all additional hypothesis are necessary.
\end{abstract}

J. E. Hutchinson \cite{hut81} considered finite families of contractions on a complete metric space and showed that they admit a unique invariant non-empty compact set $\A$. That is, $\A$ is the union of its images by those contractions. The collection of contractions is called an Iterated Function System (IFS) and the set $\A$ is its attractor. The attractor is therefore a self-similar set, to a degree which depends on the regularity of the contractions and of the separation of its images, formalized below.

M. Hata identified a condition for the connectedness and local connectedness of the attractor $\A$ in a separable complete metric space, see Theorem 4.6 of \cite{HAT}. A slightly weaker version of this result can be formulated as follows.
\begin{thm}
\label{thmHata}
The attractor $\A$ is connected if and only if it is locally connected.
\end{thm}

T. Bedford \cite{bed86}, K. J. Falconer \cite{fal89} and D. S. Mauldin  and S. Williams~\cite{mauWil88} have considered generalizations of the IFS, loosely speaking, by allowing only some puzzle pieces in the attractor. This can be seen as a projection of a sub-shift of finite type or as a Markov chain in the first two cases, as graph-directed IFS in the third. We will adopt the first point of view in this paper, see the following section for details.

One simple example of attractor of a Markov IFS is the union of attractors of two IFS, by setting the the transition matrix as full blocks corresponding to each IFS. It is therefore easy to construct counterexamples of such type for an eventual extension of Hata's result. 

We investigate in which such Markov IFS, the direct implication of Theorem \ref{thmHata} remains valid. The first obstruction is the overlapping of two IFS. Indeed, let $C$ be the standard Cantor set. The set $C \times [0,1] \se \R^2$ is the attractor of an IFS with four contractions. If one overlaps the segment $[0,1] \times \{0\}$, the resulting Markov IFS is a connected compact set that is not locally connected.

The Open Set Condition (OSC) has been introduced by P. A. P. Moran \cite{mor46}, see Theorem III. It is widely used, see for example Section 5.2 in \cite{hut81}. Computing the Hausdorff dimension of $\A$ in the absence of this property is a real challenge. This condition rules out similar examples to the one above, but only in the plane, see Section \ref{sectCE}. The OSC states the existence of a bounded open set $U$ that is forward invariant and which has disjoint images by the contractions of the IFS. 

We need a stronger version of this condition, the Homeomorphic OSC (HOSC), that is, $U$ is a Jordan domain and each contraction of the IFS should be a homeomorphism from $\ol U$ onto its image, see Definition \ref{defOSC}. The necessity of this condition is illustrated by the example below.

The sets $C \times [0, 1/2]$ and $[0,1] \times \{1\}$ can be glued without overlapping of images of the interior of the square $[0,1]^2$, by the complex map 
$$z \mapsto \exp(2\pi z).$$
This map can be used as a coordinate change of a simple Markov IFS to produce a counter-example as above.

Once these trivial obstructions are removed, we obtain the following result.

\begin{thm}
\label{thmCiLC}
If $\A$ is the connected attractor of a planar Markov IFS that satisfies HOSC, then $\A$ is locally connected.
\end{thm}

In the context of graph-directed IFS, Y. Zhang \cite{zha16} has studied the connectedness properties of the invariant sets when the contractions of the IFS are similitudes. 

\medskip
\textbf{Acknowledgement.} The author is grateful to Michael Barnsley for several enriching discussions and for asking the question that the main result of this article answers.

\section{Preliminaries}

Let $\F=(\R^d; f_1,f_2,\ldots,f_m)$ be an IFS, that is each $f_i:\R^d \ra \R^d$ is a contraction, that is a Lipschitz map with $\mathrm{Lip}(f_i)<1$. Let $I=\{1,\ldots,m\}$ and $\Si=I^\N$ be endowed with the product topology. 

Let $\pf:\Si\ra\R^d$ be the associated projection on the attractor $\A$ of $\F$. That is, for $\si\in\Si$
$$\pf(\si):=\lim_{n\ra\infty}f_{\si_0}\circ\ldots\circ f_{\si_n}(x),$$
which does not depend on $x\in\R^d$. It is straightforward to show that $\pf$ is continuous, that $\Si$ is compact and that $\pf(\Si)=\A$.

Let us also denote $\Si_n=I^n$ and for $\si\in\Si$, $\si_{|n}:=\si_0\si_1\ldots\si_{n-1}\in\Si_n$. $f_{\si_{|n}}$ denotes $f_{\si_0}\circ\ldots\circ f_{\si_{n-1}}$.

We consider subsets $\A_M$ of $\A$ that are projections of sub-shifts of finite type. That is, given a transition matrix $M\in M_m(\{0,1\})$, let  
$$\Si_M=\{x_0x_1\ldots\in\Si\ :\ M_{x_ix_{i+1}}=1, \text{ for all } i\in\N\}.$$
We set $$\A_M=\pf(\Si_M).$$
We call such a system a \emph{Markov IFS}.

\begin{definition}
\label{defOSC}
Let $d=2$. We say that $\F$ satisfies the \emph{Homeomorphic Opens Set Condition (HOSC)} if there exists a Jordan domain $U$ such that its images $f_i(U)$ are disjoint, contained in $U$ and if every map $f_i$ is a homeomorphism from $\ol U$ onto its image $f_i(\ol U)$.
\end{definition}

Let us recall the following definition, see Section I.12 in \cite{why63}. 
\begin{definition}
\label{defLc}
We say that a compact set $K\se\R^d$ is locally connected if for any $\ve>0$ there exists $\de>0$ such that for all $x,y\in K$ with $\dist(x,y)<\de$ there exists a continuum $C\se K$ containing $x$ and $y$ with $\diam C<\ve$.
\end{definition}

For any set $A \in \R^n$ and $r>0$, let us denote its $r$ neighborhood by
$$B(A,r)=\{x\in\R^n\ :\ \dist(x,A) < r\}.$$
Let us recall that that the Hausdorff distance between two non-empty sets $Y,Z$ is
$$\dist_H(Y,Z)=\inf\{r>0\ :\ Y\se B(Z,r) \text{ and }Z\se B(Y,r)\}.$$

Whenever we consider the limit of a sequence of sets, we use the Hausdorff distance. The metric space of non-empty compact sets in any $\R^n$ is complete. For any compact set $K\se\R^n$, the space of non-empty compact subsets of $K$ is also compact. 

The following result is a reformulation of Theorem I.12.1 in \cite{why63}. A \emph{continuum} is a non-empty connected compact set. A degenerate \emph{continuum} has only one point.
\begin{proposition}
\label{propLC}
If a continuum $K\se\R^d$ is not locally connected, then there there are disjoint \emph{continua} $C_i\se K$ that converge to a non-de\-ge\-ne\-rate \emph{continuum} $C\se K$ and the diameter of any \emph{continuum} joining any pair of those sets in $K$ is bounded away from zero.
\end{proposition}

Let $\ga:[a,b]\ra\ol\Om$ be a path and $\Om\se\C$ a domain, such that $\ga((a,b))\se\Om$ is a Jordan arc and $\ga(a),\ga(b)\in\pa\Om$. We call such $\ga$ a \emph{crosscut} of $\Om$.

In the sequel we assume HOSC with some Jordan domain $U$.

For some $\th\in\Si_n$, let $\Si_M^\th=\{\si\in\Si_M\ :\ \si_{|n}=\th\}$ and $\A_M^\th=\pf(\Si_M^\th)$ the corresponding puzzle piece of level $n$. For $i\in I$, we observe that $\A_M^i\se f_i(\ol U)$, that $\A_M^\th$ is homeomorphic to $\A_M^{\th_{n-1}}$ and that the attractor is the union of puzzle pieces of the first level
$$\A_M=\bigcup_{i=1}^m \A_M^i.$$
Let also 
$$U^\th=f_\th(U),$$
and observe that $\A_M^\th\se\A_M\cap \ol{U^\th}$.

\section{Proof of Theorem \ref{thmCiLC}}

Suppose $\A_M$ is connected but not locally connected. By \rpro{LC}, there is a sequence of disjoint \emph{continua} $(C_i)_{i\geq 0}$ in $\A_M$ that converge to $C\se\A_M$. We will make the following:
\begin{sa}
There exists $\th\in\Si_n$ such that $C\cap\pa U^\th\neq\es$, $C\se\ol{U^\th}$ and for all $i\geq 0$, $C_i=\ol{C_i\cap U^\th}\se \A_M^\th$. For each $C_i$, $U^\th\sm C_i$ has two connected components $L_i$ and $R_i$ such that $C\se\ol{R_i}$ and for all $j\geq 0$,
$$C_j\cap U^\th\se R_i \text{ if and only if } j > i \text{ and }C_j\cap U^\th\se L_i \text{ if and only if } j < i.$$
Also, for all $i\geq 0$, there exists a crosscut $\ga_i$ of $U^\th$, $\ga_i\se \ol{R_i\cap L_{i+1}}$ disjoint from $\A_M$, which separates $C_i$ from $C_{i+1}$ in $U^\th$.
\end{sa}

Let us remark that in the above formulation, $C$ may be degenerate. As all $C_i$ separate $U^\th$, they cannot be degenerate.

For all $i\geq 0$, let $V_i$ be the component of $R_i\cap L_{i+1}$ that contains $\ga_i$. As $V_i$ are disjoint and included in $U_\th$, their respective area tends to $0$. Thus it exists $k_\th$, such that for all $k\geq k_\th$, $V_k$ cannot contain any $U^{\th s}=f_\th\circ f_s(U)$, with $s\in I$.

By the definition of $\A_M$ and the fact that $f_\th$ is a homeomorphism from $\ol U$ to $\ol{U^\th}$, we have that  $f_\th^{-1}(\A_M^\th)\se\A_M$. The difference $\A_M\sm f_\th^{-1}(\A_M^\th)$ comes from symbols that cannot follow $\th_{n-1}$, thus excluding some puzzle pieces from $\A_M^\th$, corresponding to non empty puzzle pieces of $\A_M$. 

For all $i\geq 0$, $f_\th^{-1}(C_i)\se\A_M$ and $f_\th^{-1}(C)\se\A_M$. Those non empty sets are separated in $\ol U$ by all $f_\th^{-1}(\ga_j)$ with $j\geq i$. It is therefore enough to show that there exists $k>0$ such that 
$$f_\th^{-1}(\ga_k)\cap\A_M=\es,$$ 
to prove that $\A_M$ is disconnected, a contradiction.

Indeed, let $k\geq k_\th$ and suppose there exists $s\in I$ and 
$$x\in f_\th^{-1}(\ga_k)\cap\A^s_M.$$
As $y:=f_\th(x)\in\ga_k$ which is disjoint from $\A_M^\th$, $y\in \ol{U^{\th s}}\sm \A_M^\th$, with $M_{\th_{n-1}s}=0$. As $V_k \supseteq \ga_k$ cannot contain $U^{\th s}\se U^\th$, we have either 
$$V_k\cap U^{\th s}=\es \text{ or } U^{\th s}\cap \pa V_k\neq\es.$$
Suppose for now we are in the latter case. As $V_k$ is a connected component of $U^\th\sm(C_k\cup C_{k+1})$, $\pa V_k\se C_k\cup C_{k+1}\cup\pa U^\th$. Also, $U^{\th s}\se U^\th$, therefore $U^{\th s}\cap\pa U^\th=\es$. We obtain that 
$$U^{\th s}\cap\A_M^\th\neq\es,$$
which contradicts $M_{\th_{n-1}s}=0$, because $\ol{U^{\th i}},i\in I \sm \{s\}$, are disjoint from $U^{\th s}$, by the HOSC.

The only possibility left is $V_k\cap U^{\th s}=\es$ so $y\in \pa V_k\cap \pa U^{\th s}$ and $y$ is an endpoint of $\ga_k$, therefore $y\in\pa U^\th$. Recall that $\ol {U^\th}$ is homeomorphic to $\ol U$ and thus to the closed unit disk. As $U^{\th s}$ is disjoint from $C_k\cup C_{k+1}\se\A_M$ and from $V_k$, $U^{\th s}$ is separated from $V_k$ in $U^\th$ by $C_k\cup C_{k+1}$. We conclude that 
$$y\in C_k\cup C_{k+1}\se \A_M^\th,$$
a contradiction.

\section{Proof of the Simplifying assumption}

A compact Hausdorff topological space $X$ is normal, that is, disjoint closed sets can be included in disjoint open sets. 

The \emph{connected component} of $x\in X$ (also referred as simply the component of $x$) is the union of all connected subspaces of $X$ containing $x$, a connected closed set. Connected components form a partition of the space $X$. 

The \emph{quasi-component} of $x$ is the intersection of all open and closed sets (also called \emph{clopen}) containing $x$. 

In general, we only have that the connected component of $x$ is included in its quasi-component (Theorem 6.1.22 in \cite{GT}). If the space is compact Hausdorff, we have equality. Let us cite  this result, Theorem 6.1.23 in \cite{GT}.

\begin{thm}
\label{thmQCEC}
In a compact Hausdorff space $X$, the component of a point $x\in X$ coincides with the quasi-component of the point $x$.
\end{thm}

Let us prove the following separation result.

\begin{proposition}
\label{propSep}
Let $K\se \C$ be a planar compact set and $x,y\in K$ which are not contained in the same component of $K$. There exists an analytic Jordan curve disjoint from $K$ that separates $x$ and $y$.
\end{proposition}

\begin{proof}
Using the previous theorem, there are two compact sets $K_1\ni x$ and $K_2\ni y$ which form a partition of $K$. Let $d>0$ be the distance between $K_1$ and $K_2$ and $U_1$ and $U_2$ the disjoint $\frac d3$-neighborhoods of $K_1$ and respectively of $K_2$. Let $U_x$ be the component of $U_1$ containing $x$ and $U_y$ be the component of $U_1$ containing $y$.

We may assume, up to a permutation of $x$ and $y$, that $U_y$ is contained in the unbounded component $C_x$ of $U_x^c:=\C\sm U_x$. Let also $V_x$ be the \emph{filled in} $U_x$, that is $V_x=C_x^c$. As both $V_x$ and $V_x^c$ are connected, $V_c$ is simply connected.

Remark that $\pa V_x\se \pa U_x$, therefore $\pa V_x \cap K=\es$, as $\pa U_x$ is at distance $\frac d3$ from $K$. Then $K_x:=V_x \cap K$ is compact. Let $\phi:V_x\ra\D$ be a Riemann mapping. Then $\phi(K_x)$ is compact so there is $0<r<1$ such that $\phi(K_x)\se D(0,r)$, the disk of radius $r$ centered in the origin. 

Observe that $\phi^{-1}(\pa D(0,r))$ is an analytic Jordan curve, disjoint from $K$, that separates $x\in K_x$ from $C_x$, which contains $U_y\ni y$.
\end{proof}

\begin{lemma}
\label{lemCon}
If a compact set is the limit of a sequence of compact connected sets, then it is connected.
\end{lemma}
\begin{proof}
Suppose the limit set $K$ is disconnected, then it admits a partition into two compact sets $C_1$ and $C_2$. Let $d=\dist(C_1,C_2)>0$ be the euclidian distance between the two sets. There exists a set $K_i$ in the sequence such that the Hausdorff distance
$$\dist_H(K,K_i) < \frac d3.$$
Then $K_i\se B(C_1,\frac d3)\cup B(C_2,\frac d3)$ and it intersects both of these disjoint neighborhoods, which contradicts the fact that $K_i$ is connected.
\end{proof}

Let $A,B\se \C$ and $A$ be connected. We say that $x$ and $y$ are \emph{separated by $B$ in $A$} if they are contained in distinct components of $A\sm B$. In the sequel of this section, all considered sets are planar sets.

We will need the following separation results.

\begin{lemma}
\label{lemSep}
Let $U$ be a Jordan domain and $K\se\ol U$ a continuum. Any two disjoint connected components of $\pa U\sm K$ are separated by $K$ in $\ol U$.
\end{lemma}
\begin{proof}
Let $a,b\in \pa U$ be contained in disjoint components of $\pa U\sm K$, thus $K$ intersects both open Jordan arcs $\arc{ab}$ and $\arc{ba}$. As $\ol U$ is homeomorphic to the closed unit disk, it is locally arc connected. As every domain is arc connected, if $a$ and $b$ are in the same connected component of $\ol U\sm K$, then they are connected by a Jordan arc $\ga$ disjoint from $K$. 

Jordan curves $\ga\, \arc{ba}$ and $\arc{ab}\,\ol\ga$ bound the two components $U_1,U_2$ of $\ol U\sm \ga$, each containing a point of $K$. As $K\se U_1\cup U_2$, $U_1 \cap U_2=\es$ and they are open in $\ol U$, this contradicts the connectedness of $K$.
\end{proof}

\begin{lemma}
\label{lemCc}
Let $U$ be a Jordan domain and $K$ be a continuum which is not contained in $U$. Let $K':=K \cap \ol U$. Then $K' \cup \pa U$ is connected and all connected components $C$ of $K'$ intersect $\pa U$.
\end{lemma}
\begin{proof}
If $K \cap U=\es$, there is nothing to prove. Otherwise, $K \cap \pa U \neq \es$, by the connectedness of $K$.

Assume that $K' \cup \pa U$ is disconnected, then by Theorem \ref{thmQCEC}, there are two non-empty compact sets $K_1$ and $K_2$ which form a partition of $K' \cup \pa U$. As $\pa U$ is connected, we may assume that $\pa U \se K_1$, and therefore $K_2 \se U$. Then $K_1 \cup (K \sm U)$ and $K_2$ are compact sets and form a partition of the connected set $K \cup \pa U$, a contradiction.

Similarly, let $C$ be a connected component of $K'$ with $C \cap \pa U = \es$. Then $C$ and $(K \cup \pa U) \sm C$ partition of $K \cup \pa U$ into compact sets, a contradiction.
\end{proof}

\begin{proposition}
\label{propCc}
Let $U$ be a Jordan domain and $K$ a continuum that separates $x$ from $y$ in $U$. Then there exists a connected component of $K\cap\ol U$ that separates $x$ from $y$ in $U$.
\end{proposition}

\begin{proof} Let $U_x$ and $U_y$ be the disjoint components of $U\sm K$ containing $x$, and respectively $y$. If $\pa U_x\se K$, then the connected component of $K':=K\cap\ol U$ containing $\pa U_x$ separates $x$ from $y$ in $U$. Therefore we may assume that both $U_x$ and $U_y$ have common boundary points with $U$, outside $K$. Let us denote two such points by $a$ and respectively $b$. Observe that $K \nsubseteq U$.

By Lemma \ref{lemCc}, every component of $K'$ intersects at least one of the Jordan arcs $\arc {ab}$ or $\arc {ba}$. If there is such a component $C$ that intersects both arcs, $x$ and $y$ are separated by $C$ in $U$ by Lemma \ref{lemSep}. We find such a component to complete the proof.

Let $K_1$ and $K_2$ be the unions of components of $K'$ that intersect $\arc {ab}$ and respectively $\arc {ba}$. Let $\ol{ab}$ be the minimal sub-arc of $\arc{ab}$ that contains $K\cap\arc{ab}$, a closed arc. Let $K_1'=K_1\cup \ol{ab}$, a connected set. Let also $K_2'$ be constructed in a similar way. 

If 
$$\dist(K_1',K_2')>0,$$
then by Proposition \ref{propSep}, there is a Jordan curve $\ga$ that separates $\ol{K_1'}$ from $\ol{K_2'}$ in the plane. Therefore $\ga$ intersects both components of $\pa U\sm K$ containing $a$ and respectively $b$. We can then find a sub-arc that connects those components inside $U$, as on a circle containing two disjoint closed sets $A$ and $B$, we can always find an arc in $(A\cup B)^c$ that connects a point of $A$ to a point of $B$. This contradicts the fact that $a$ and $b$ are not in the same component of $\ol U\sm K$.

If 
$$\dist(K_1',K_2')=0,$$
then, up to a permutation of $K_1$ and $K_2$, there is a sequence of components $C_i$ of $K_1$ that have a limit point $c$ in $K_2'$. As $C_i$ are compact, we may assume they have a limit $C\se K$, a continuum, by Lemma \ref{lemCon}. Then $c\in C$, thus the component $C'$ of $K\cap \ol U$ containing $C$ intersects $\arc{ba}$. Also, each $C_i$ contains at least a point in $\ol{ab}$, therefore $C$ intersects $\ol{ab}$. 

The continuum $C'$ intersects both arcs $\arc {ab}$ and $\arc {ba}$, which completes the proof.
\end{proof}

Remark that we can replace $x$ and $y$ in the statement of the previous proposition by any connected subsets of $U\sm K$. 

\medskip
\noindent\textbf{Global version of the simplifying assumption.} As a first step of our construction, we will drop the condition that the Jordan domain that is separated by the sequence of continua $(C_i)_{i\geq 0}$ corresponds to a puzzle piece. That is, we replace $U^\th$ by some Jordan domain $\Om$ in the statement of the \emph{simplifying assumption}.

\begin{figure}[h]
\label{figGSA}
\centering{
\caption{Global version of the simplifying assumption}
\includegraphics[width=11cm, height=8cm]{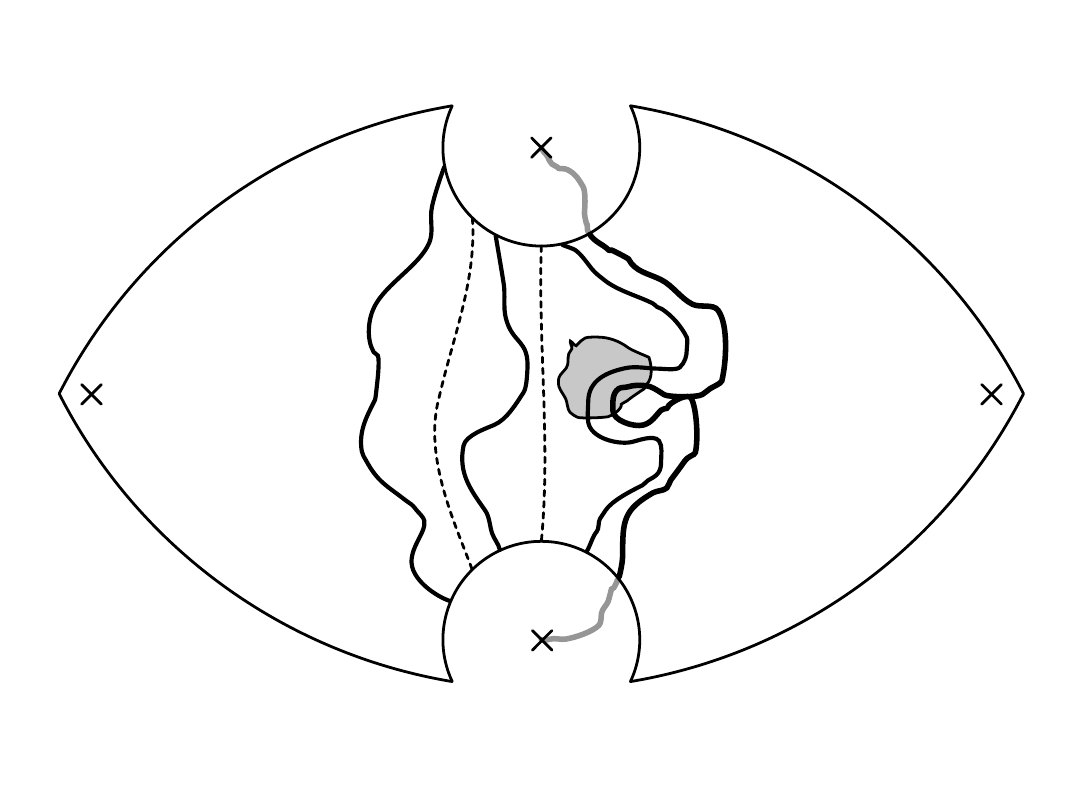} 
\setlength{\unitlength}{1mm}

\begin{picture}(-110,0)(110,0) 
\put(51,19){$b$}
\put(51,69){$a$}
\put(12,44){$c$}
\put(22,44){$L_0$}
\put(96,44){$d$}
\put(17,67){$\Om$}
\put(37,60){$C_0'$}
\put(48,37){$C_1'$}
\put(59,36.25){$C_2'$}
\put(75,50){$C'$}
\put(41,50){$\ga_0$}
\put(56,33){$\ga_1$}
\put(57,52){$U^\th$}
\put(58,69){$D_a$}
\put(53,24){$D_b$}
\end{picture}
}
\end{figure}
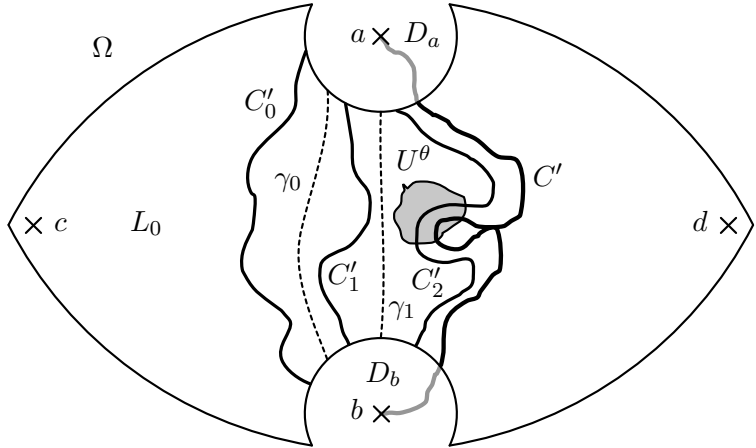

By Proposition \ref{propLC}, there is a sequence of continua $(C_i)_{i\geq 0}\se\A_M$ that converges to $C_\infty:=C\se\A_M$ such that $d:=\diam(C)>0$ and any two sets $C_i\neq C_j$ (including $C_\infty$), cannot be connected inside $\A_M$ by a continuum of diameter smaller than $10d$. Also, by extracting a subsequence if need be, we can assume that for all $i\geq 0$,
$$\dist_H(C_i,C)<\frac d{20}.$$

Let $a,b\in C$ be at maximal distance $\dist(a,b)=d$, $D_a=B(a,\frac d{5})$, $D_b=B(b,\frac d{5})$, $\Om'=B(a, \frac{11}{10}d)\cap B(b, \frac{11}{10}d)$ and 
$$\Om=\Om'\sm (\ol{D_a}\cup\ol{D_b}),$$
a Jordan domain.

All $C_i,i\in\ol\N:=\N \cup \{\infty\}$ are included in $\Om'$ and have common points with both $D_a$ and $D_b$. As they are connected, they all have common points with both $\pa D_a \cap \pa \Om$ and $\pa D_b \cap \pa \Om$. Let $\{c,d\}=\pa B(a, \frac{21}{20}d) \cap \pa B(b, \frac{21}{20}d)$, disjoint from all $C_i,i\in\ol\N$. 

It follows from Lemma \ref{lemSep} applied to 
$
K:=(C_i\cup\pa D_a \cup \pa D_b) \cap \ol\Om
$
that every $C_i$ separates $c$ from $d$ in $\Om$. Proposition \ref{propCc} provides a connected component $C_i'$ of $C_i\cap \ol\Om$ that separates $c$ from $d$ in $\Om$. 

Let $C_\infty'=C'\se \ol\Om$ be the limit of the sequence $(C_i')_{i\geq 0}$, a connected set by Lemma \ref{lemCon} that separates $c$ from $d$ in $\Om$ by Lemma \ref{lemSep}, as it has points in both $\pa D_a \cap \pa \Om$ and $\pa D_b \cap \pa \Om$.

We can easily obtain the following.
\begin{lemma}
\label{lemNS}
If $i,j\in\ol\N$ are distinct, then $C_i'$ is not separated from both $c$ and $d$ in $\Om$ by $C_j'$.
\end{lemma}
\begin{proof} Assume the contrary. Then $C_j'$ has points in both components of $c$ and of $d$ in $\ol\Om\sm C_i'$. As $C_i'\cap C_j'=\es$, by Lemma \ref{lemSep}, this implies that $C_j'$ is disconnected, a contradiction.
\end{proof}

We can therefore introduce a total order on $\{C_i'\ :\ i\in\ol\N\}$ and we say that $C_i'\leq C_j'$ if and only if $C_j'$ does not separate $C_i'$ from $c$ in $\Om$. As a Hausdorff limit, in between any $C_i'$ and $C'$ (either for $C_i'\leq C'$ or for $C_i'\geq C'$), there is some $C_j'$. Then, up to extracting a subsequence and permuting $c$ and $d$, we may assume that 
$$\forall i,j\in\ol\N,\ C_i'\leq C_j' \text{ if and only if }i\leq j.$$

By Proposition \ref{propSep}, for all $C_i',i\in\N$, there exist a Jordan curve $\al_i$ disjoint from $\A_M\cap \ol\Om$ that separates $C_i'$ from $C'$. Again, by extracting a subsequence if need be, we may assume that $\al_i$ separates $C_i'$ from $C_{i+1}'$. Proposition \ref{propCc} guarantees the existence of a crosscut $\ga_i$, an arc of $\al_i$, with endpoints in $\pa D_a$ and $\pa D_b$, that separates $C_i'$ from $C_{i+1}'$ in $\Om$. Let us denote by $L_i$ the component of $c$ in $\Om\sm C_i'$ and by $R_i$ the component of $d$ in $\Om\sm C_i'$.

We have proven the global version of the \emph{simplifying assumption}. We conclude the proof by showing that $\Om$ can be replaced by some $U^\th$, with $\th\in\Si_n$.

\begin{proof}[\textbf{Proof of the simplifying assumption}] 
We find some $U^\th\se\Om$ such that $\ol{U^\th}\cap C'\neq\es$ and that $C_i\cap U^\th\neq\es$ for infinitely many $i\in\N$. As $\diam(U^\th)$ converges to $0$ uniformly when the length $n$ of $\th$ goes to infinity, we can fix $n>0$ such that for all $\th\in\Si_n$,
$$\diam(U_\th)<\frac d{100}.$$
Let $L=\bigcup_{i\geq 0}L_i$. Let us recall that all $C_i'\se\A_M\se \bigcup_{\th\in\Si_n}\ol{U^\th}$ and that $\Si_n$ is finite. Also, for all $i\in\N,C_i'\se L$ and their Hausdorff limit $C'\se\pa L$. 

Therefore there exists $\th\in\Si_n$ such that $U^\th\se\Om$, $U^\th\cap L\neq\es$ and $\dist(U^\th, C')=0$, thus $\ol{U^\th}\cap C'\neq\es$. Fix $y\in C'\cap\ol{U^\th}$ and $x\in U^\th\cap L$. 
There exists $k_x$ such that $x\in L_{k_x}$, thus it is separated in $\Om$ from $y$ by all $C_i'$ with $i\geq k_x$. As $U^\th\se\Om$, the two points are also separated by all $C_i'$ with $i\geq k_x$ in $\ol{U^\th}$. Thus by Proposition \ref{propCc}, for all $i\geq k_x$, a component $K_i$ of $\ol{U^\th} \cap C_i'$ separates $x$ from $y$ in $\ol{U^\th}$. Thus $K_{i+1}$ separates $K_i$ from $y$ in $\ol{U^\th}$.

Let $K$ be the Hausdorff limit of $(K_i)_{i\geq 0}$, passing to a subsequence if necessary. By Lemma \ref{lemCon}, $K\se\ol{U^\th}\cap C'$ is connected. Similarly, by considering sub-arcs, we may assume by Proposition \ref{propCc} that $\ga_i$ are crosscuts of $U^\th$ separating $K_i$ from $K_{i+1}$ in $\ol{U^\th}$. The order on $C_i'$ is inherited by the sets $K_i$. We set $L_i'$ and $R_i'$ to be the components of $U^\th\sm K_i$ containing $K_{i-1}$ and respectively $y$.

We have proven the \emph{simplifying assumption} and, as a consequence, Theorem \ref{thmCiLC}.
\end{proof}

\section{Counterexample in $\R^3$}
\label{sectCE}

In the plane, connected sets locally separate the plane and bound open sets, and thus their area.
This fact plays a central rôle in the proof of Theorem~\ref{thmCiLC}. In higher dimension, these relations break down, and the result does not hold anymore. The following counterexample in $\R^3$ sheds a more light on the proof of Theorem \ref{thmCiLC}.  

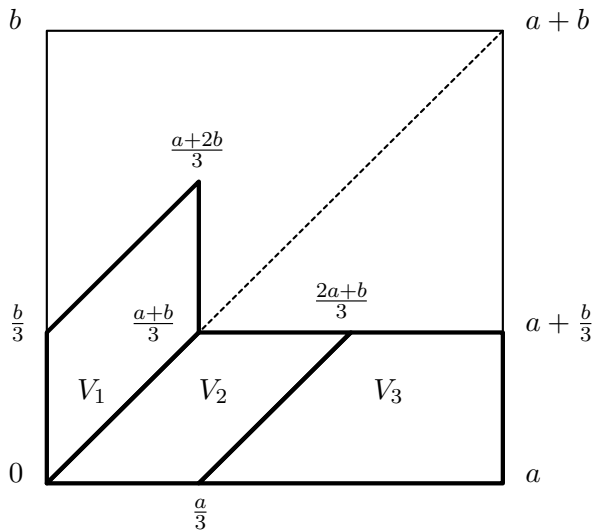
\begin{figure}[h]
\label{figCE}
\centering{
\caption{Counterexample in $\R^3$}
\includegraphics[width=8cm, height=8cm]{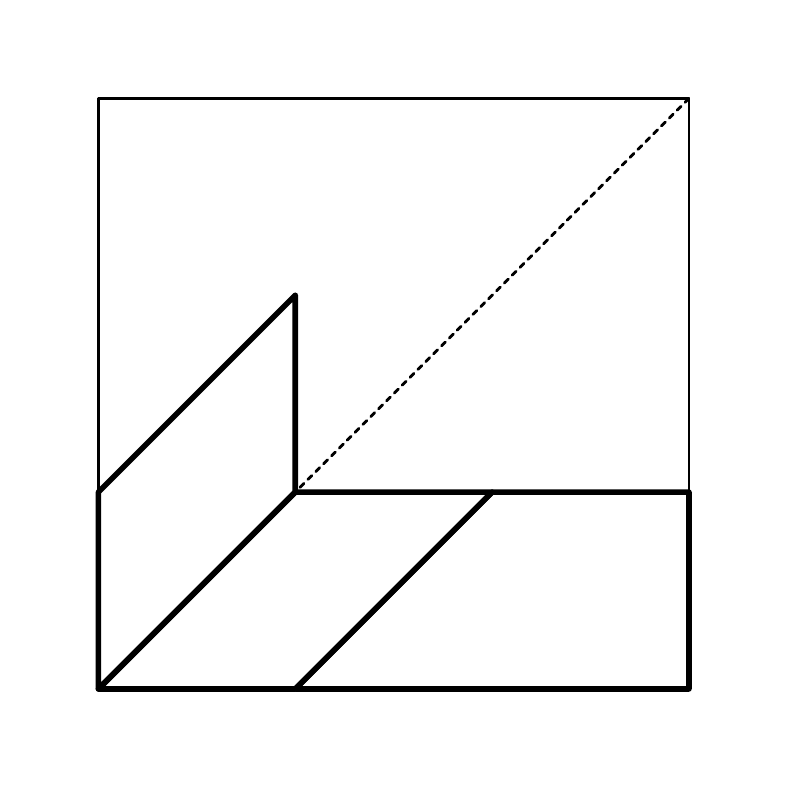} 
\setlength{\unitlength}{1mm}

\begin{picture}(-80,0)(80,0) 
\put(5,15){$0$}
\put(5,75){$b$}
\put(73,15){$a$}
\put(73,75){$a+b$}
\put(5,35){$\frac{b}3$}
\put(21,35){$\frac{a + b}3$}
\put(45,38){$\frac{2a + b}3$}
\put(73,35){$a + \frac{b}3$}
\put(29,10){$\frac{a}3$}
\put(26,58){$\frac{a + 2b}3$}
\put(14,26){$V_1$}
\put(30,26){$V_2$}
\put(53,26){$V_3$}
\end{picture}
}
\end{figure}

Let $a=(1;0)$ and $b=(0;1)$ form the canonical orthonormal basis of $\R^2$. Let $g_1,g_2\in L(\R^2)$ be linear maps such that $g_1(a)=\frac{a+b}3, g_1(b)=\frac b3$ and $g_2(a)=\frac a3, g_2(b)=\frac{a+b}3$. $g_1$ and $g_2$ are contracting homeomorphisms of $\R^2$ sending the unit square $Q=[0,1]^2$ onto $V_1$ and respectively $V_2$, domains bounded by polygons $(0,\frac b3, \frac{a+2b}3, \frac{a+b}3)$ and respectively $(0,\frac {a+b}3, \frac{2a+b}3, \frac{a}3)$, as illustrated by Figure \ref{figCE}. 

It is not hard to check that there is a contracting homeomorphism $g_3$ of the plane that sends $Q$ onto $V_3$, the interior of the polygon $(\frac a3, \frac{2a+b}3, a+\frac{b}3, a)$, vertices which are the images by $g_3$ of $0,b,a+b$ and respectively $a$. We may also assume that the restriction of $g_3$ on $\R\times\{0\}$ is affine.

Let us consider two affine maps $h_1,h_2:\R\ra\R$ such that for all $x\in\R$, $h_1(x)=\frac x2$ and $h_2(x)=\frac{x+1}2$. We can now construct contracting homeomorphisms of $\R^3$, $f_i,i=1\ldots 6$ given by $f_1=g_1\times h_1$, $f_2=g_1\times h_2$, $f_3=g_2\times h_1$, $f_4=g_3\times h_1$, $f_5=g_2\times h_2$ and $f_6=g_3\times h_2$.

We observe that contractions $f_i, i=1\ldots 6$ satisfy the $\R^3$ version of OSC, if we set $U=(0,1)^3$ the open unit cube. All contractions $f_i$ being homeomorphisms of $\R^3$ is a natural extension of working with Jordan domains in the plane, as any homeomorphism of a closed Jordan domain can be extended to a homeomorphism of the plane.

Observe that the attractor of the IFS $\{f_1,f_2\}$ is $L_1=\{(0;0)\}\times[0,1]$. Also, the attractor of $\{f_3,f_4\}$ is $L_2=[0,1]\times\{(0;0)\}$ and the attractor of $\{f_5,f_6\}$ is $L_3=[0,1]\times\{(0;1)\}$. 

Consider the Markov IFS $\F=\{\R^3;f_1,\ldots,f_6;M\}$ with transition matrix
 $$M=\left(
\begin{array}{cccccc}
1 & 1 & 0 & 0 & 0 & 0 \\
1 & 1 & 0 & 0 & 0 & 0 \\
0 & 0 & 1 & 1 & 1 & 1 \\
0 & 0 & 1 & 1 & 1 & 1 \\
0 & 0 & 0 & 0 & 1 & 1 \\
0 & 0 & 0 & 0 & 1 & 1 \\
\end{array}
\right).$$

As $\F$ contains the aforementioned IFS consisting of two maps each, it contains their respective attractors, the segments $L_1,L_2$ and $L_3$. Also, as $M_{3,5}=M_{3,6}=M_{4,5}=M_{4,6}=1$, it also contains a copy of $L_3$, more precisely $[0,1]\times \{(0;\frac 12)\}$. By induction, one can obtain that the attractor of $\F$ is
$$\A_M=L_1\cup [0,1]\times \{0\}\times \left\{0,\ldots,2^{-n},\ldots,\frac 14, \frac 12, 1\right\},$$
which is connected, but not locally connected.


\bibliography{bib}
\bibliographystyle{plain}

\end{document}